\documentclass[12pt]{article}
\usepackage[utf8]{inputenc}
\usepackage[T1]{fontenc}
\usepackage{amsmath,amssymb,amsthm}
\usepackage{geometry}
\usepackage{hyperref}
\usepackage{indentfirst}
\usepackage{booktabs}
\usepackage{xcolor}
\usepackage{fvextra}
\DefineVerbatimEnvironment{code}{Verbatim}{
  breaklines=true,
  breakanywhere=false,
  fontsize=\footnotesize,
  xleftmargin=1em,
}
\DefineVerbatimEnvironment{term}{Verbatim}{
  breaklines=true,
  breakanywhere=true,
  fontsize=\footnotesize,
  xleftmargin=1em,
}

\newtheorem{theorem}{Theorem}
\newtheorem{lemma}{Lemma}
\newtheorem{proposition}{Proposition}

\theoremstyle{remark}
\newtheorem{remark}{Remark}
\theoremstyle{definition}
\newtheorem{definition}{Definition}

\title{Non-Existence of Linear--Quartic Factorization for the Second Cuboid Quintic}

\author{Valery Asiryan\\[3pt]
\small \texttt{asiryanvalery@gmail.com}}
\date{\small January 5, 2026}

\begin{document}

\maketitle

\begin{abstract}
Let $Q_{p,q}(t)\in\mathbb{Z}[t]$ be Sharipov's even monic degree-$10$ \emph{second cuboid polynomial} depending on coprime integers $p\neq q>0$.
Writing $Q_{p,q}(t)$ as a quintic in $t^{2}$ produces an associated monic quintic polynomial.
After the weighted normalization $r=p/q$ and $s=r^{2}$ we obtain a one-parameter family $P_s(x)\in\mathbb{Q}[x]$ such that
\[
Q_{p,q}(t)=q^{20}\,P_s\!\left(\frac{t^{2}}{q^{4}}\right)\qquad\text{with}\qquad s=\left(\frac{p}{q}\right)^{2}.
\]
We show that for every rational $s>0$ with $s\neq 1$ the equation $P_s(x)=0$ has no rational solutions.
Equivalently, $P_s$ admits no $1+4$ factorization over $\mathbb{Q}$.

The proof uses an explicit quotient by the inversion involution $(s,y)\mapsto(1/s,1/y)$ and reduces the rational-root problem for $P_s$ to rational points on the fixed genus-$2$ hyperelliptic curve
\[
C:\quad w^2=t^5+21t^4+26t^3+10t^2+5t+1=(t+1)(t^4+20t^3+6t^2+4t+1).
\]
Using \textsc{Magma} and Chabauty's method on the Jacobian of $C$, we compute $C(\mathbb{Q})$ exactly and conclude that the only parameter value producing a rational root is the excluded case $s=1$ (equivalently $p=q$).

As a consequence, for coprime $p\neq q>0$ the polynomial $Q_{p,q}(t)$ has no rational roots (hence no linear factor over $\mathbb{Q}$, and in particular no linear factor over $\mathbb{Z}$).


\end{abstract}


\section{Introduction}\label{sec:intro}

The perfect cuboid problem asks for a rectangular box with integer edges such that all three face diagonals and the space diagonal are integers.
In a framework due to R.\,A.\,Sharipov \cite{Sharipov2011Cuboids,Sharipov2011Note}, one is led to explicit parameter-dependent even polynomials whose irreducibility is conjectured for coprime parameters.
In particular, Sharipov defines the \emph{second cuboid polynomial} $Q_{p,q}(t)\in\mathbb{Z}[t]$ of degree $10$ and formulates the conjecture that $Q_{p,q}(t)$ is irreducible over $\mathbb{Z}$ for coprime $p\neq q>0$.

\medskip
\noindent\textbf{Goal.}
We close the \emph{linear--quartic} case for the normalized associated quintic $P_s(x)$:
we prove that for every rational $s>0$ with $s\neq 1$ the quintic $P_s$ has no rational root (equivalently, it admits no $1+4$ factorization over $\mathbb{Q}$).
As a consequence, for coprime $p\neq q>0$ the second cuboid polynomial $Q_{p,q}(t)$ has no rational roots (and hence no linear factor).

\medskip
\noindent\textbf{Method.}
The key observation is an explicit inversion symmetry and a quotient reduction that transforms the rational-root condition for a one-parameter family of quintics into the computation of rational points on a fixed genus-$2$ hyperelliptic curve.
We then compute all rational points on that curve using \textsc{Magma}~\cite{Magma} (rank bound plus Chabauty on the Jacobian), yielding a certificate-style closure of the $1+4$ case.


\section{The second cuboid polynomial and its associated quintic}\label{sec:poly}

\subsection{Sharipov's second cuboid polynomial \texorpdfstring{$Q_{p,q}(t)$}{Qp,q(t)}}
Let $p,q\in\mathbb{Z}_{>0}$ be coprime and $p\neq q$.
The second cuboid polynomial is the even monic degree-$10$ polynomial
\begin{align}
Q_{p,q}(t)
={}&t^{10} + (2 q^{2} + p^{2})(3 q^{2} - 2 p^{2})\, t^{8} \nonumber\\
&+ (q^{8} + 10 p^{2} q^{6} + 4 p^{4} q^{4} - 14 p^{6} q^{2} + p^{8})\, t^{6} \nonumber\\
&- p^{2} q^{2}\,(q^{8} - 14 p^{2} q^{6} + 4 p^{4} q^{4} + 10 p^{6} q^{2} + p^{8})\, t^{4} \nonumber\\
&- p^{6} q^{6}\,(q^{2} + 2 p^{2})(-2 q^{2} + 3 p^{2})\, t^{2}
- p^{10} q^{10}\in\mathbb{Z}[t].\label{eq:Qpq}
\end{align}
This is the polynomial denoted $Q_{p,q}(t)$ in~\cite{Sharipov2011Cuboids,Sharipov2011Note}.

\subsection{Weighted normalization to \texorpdfstring{$Q_r(u)$}{Qr(u)}}
The polynomial \eqref{eq:Qpq} is weighted-homogeneous of total weight $20$ for
\[
\deg(p)=\deg(q)=1,\qquad \deg(t)=2,
\]
hence one may normalize to a one-parameter family.

\begin{lemma}[Normalization]\label{lem:normalize}
Let $q\neq 0$ and set
\[
r:=\frac{p}{q}\in\mathbb{Q},\qquad u:=\frac{t}{q^{2}}\in\mathbb{Q}.
\]
Then
\begin{equation}\label{eq:normalize}
Q_{p,q}(t)=q^{20}\,Q_r(u),
\end{equation}
where
\begin{equation}\label{eq:Qr}
\begin{aligned}
Q_r(u)
={}& u^{10} + (2+r^2)(3-2r^2)\,u^{8} + \bigl(1+10r^2+4r^4-14r^6+r^8\bigr)\,u^{6}\\
& - r^2\bigl(1-14r^2+4r^4+10r^6+r^8\bigr)\,u^{4}
 - r^6(1+2r^2)(-2+3r^2)\,u^{2}
 - r^{10}\in\mathbb{Q}[u].
\end{aligned}
\end{equation}
\end{lemma}

\begin{proof}
Substitute $p=rq$ and $t=q^2u$ into \eqref{eq:Qpq} and factor out $q^{20}$.
\end{proof}

\subsection{The associated quintic \texorpdfstring{$P_s(x)$}{Ps(x)}}
Since $Q_r(u)$ is even, it is a quintic in $x=u^2$.

\begin{definition}[Second cuboid quintic]\label{def:Ps}
Let $s:=r^2\in\mathbb{Q}_{\ge 0}$.
Define $P_s(x)\in\mathbb{Q}[x]$ by the identity
\begin{equation}\label{eq:QrPs}
Q_r(u)=P_s(u^2).
\end{equation}
Equivalently, $P_s(x)$ is the monic quintic
\begin{equation}\label{eq:Ps}
\begin{aligned}
P_s(x)
={}&x^{5} + (2+s)(3-2s)\,x^{4} + (1+10s+4s^{2}-14s^{3}+s^{4})\,x^{3}\\
& - s(1-14s+4s^{2}+10s^{3}+s^{4})\,x^{2}
 - s^{3}(1+2s)(-2+3s)\,x
 - s^{5}.
\end{aligned}
\end{equation}
\end{definition}

\begin{proof}[Derivation]
Substitute $x=u^2$ into \eqref{eq:Qr} and set $s=r^2$.
\end{proof}

\begin{definition}[$1+4$ factorization for $P_s$]\label{def:14}
Let $K$ be a field of characteristic $0$.
We say that a monic quintic $P(x)\in K[x]$ admits a \emph{$1+4$ factorization over $K$} if it has a root in $K$ (equivalently $P(x)=(x-x_0)H(x)$ with $\deg H=4$).
\end{definition}

\begin{lemma}[Roots and even quadratic factors]\label{lem:14_28}
Let $p,q\in\mathbb{Z}_{>0}$, $q\neq 0$, and set $r=p/q$ and $s=r^2$.
Then the following are equivalent:
\begin{enumerate}
\item $P_s(x)$ has a rational root $x_0\in\mathbb{Q}$;
\item $Q_r(u)$ is divisible in $\mathbb{Q}[u]$ by the even quadratic factor $u^{2}-x_0$;
\item $Q_{p,q}(t)$ is divisible in $\mathbb{Q}[t]$ by the even quadratic factor $t^{2}-q^{4}x_0$.
\end{enumerate}
Moreover, if $Q_{p,q}(t)$ has a rational root $t_0\in\mathbb{Q}$, then $P_s$ has a rational root $x_0=(t_0/q^{2})^{2}\in\mathbb{Q}$.
\end{lemma}

\begin{proof}
The equivalence of (1) and (2) follows from \eqref{eq:QrPs}.
Indeed, if $P_s(x_0)=0$ then $(x-x_0)\mid P_s(x)$ in $\mathbb{Q}[x]$, hence $(u^2-x_0)\mid P_s(u^2)=Q_r(u)$ in $\mathbb{Q}[u]$.
Conversely, if $(u^2-x_0)\mid P_s(u^2)$, then in the quotient ring $\mathbb{Q}[u]/(u^2-x_0)$ we have $P_s(u^2)=P_s(x_0)=0$, hence $P_s(x_0)=0$ in $\mathbb{Q}$.

The equivalence of (2) and (3) follows from the normalization \eqref{eq:normalize}:
substitute $u=t/q^2$ and clear denominators to see that $u^2-x_0$ divides $Q_r(u)$ if and only if $t^2-q^4x_0$ divides $Q_{p,q}(t)$.

For the final claim, if $Q_{p,q}(t_0)=0$ with $t_0\in\mathbb{Q}$ then $Q_r(u_0)=0$ for $u_0=t_0/q^2\in\mathbb{Q}$, hence $P_s(u_0^2)=0$ and $x_0=u_0^2=(t_0/q^2)^2$ is a rational root of $P_s$.
\end{proof}

\begin{remark}\label{rem:domain}
In the original cuboid setting we have $p,q>0$, hence $r>0$ and $s=r^2\in\mathbb{Q}_{>0}$.
Moreover $p\neq q$ is equivalent to $r\neq 1$, i.e.\ $s\neq 1$.
\end{remark}


\section{Normalized root equation and inversion symmetry}\label{sec:involution}

Fix $s\in\mathbb{Q}_{>0}$.
A $1+4$ factorization of $P_s$ is equivalent to the existence of $x\in\mathbb{Q}$ such that $P_s(x)=0$.
Since $P_s(0)=-s^5\neq 0$ for $s\neq 0$, any root $x$ satisfies $x\neq 0$.
It is convenient to scale by $s$.

\begin{lemma}[Scaled root equation]\label{lem:F}
Let $s\in\mathbb{Q}\setminus\{0\}$ and set $x=s y$.
Then
\[
P_s(x)=0\quad\Longleftrightarrow\quad F(s,y)=0,
\]
where $F(s,y)\in\mathbb{Z}[s,y]$ is
\begin{equation}\label{eq:F}
\begin{aligned}
F(s,y)=\;& s^{2}y^{5} + (-2s^{3}-s^{2}+6s)y^{4} + (s^{4}-14s^{3}+4s^{2}+10s+1)y^{3}\\
&+(-s^{4}-10s^{3}-4s^{2}+14s-1)y^{2} + (-6s^{3}+s^{2}+2s)y - s^{2}.
\end{aligned}
\end{equation}
\end{lemma}

\begin{proof}
Substitute $x=sy$ into \eqref{eq:Ps} and divide the resulting identity by $s^3$ (valid for $s\neq 0$).
\end{proof}

\begin{lemma}[Inversion symmetry]\label{lem:inv}
The curve $F(s,y)=0$ is invariant under the involution
\[
(s,y)\longmapsto\left(\frac1s,\frac1y\right).
\]
More precisely,
\begin{equation}\label{eq:inv}
F\!\left(\frac1s,\frac1y\right)=-\frac{1}{s^{4}y^{5}}\,F(s,y).
\end{equation}
\end{lemma}

\begin{proof}
This is a direct verification from \eqref{eq:F} by substitution and simplification.
\end{proof}


\section{Quotient by inversion and a rational parametrization}\label{sec:quotient}

By Lemma~\ref{lem:inv}, it is natural to pass to invariants of the inversion involution.

\begin{definition}[Invariants]\label{def:UV}
Define
\[
U:=s+\frac1s,\qquad V:=y+\frac1y.
\]
Equivalently, $s$ and $y$ satisfy the quadratics
\begin{equation}\label{eq:quadUV}
s^{2}-Us+1=0,\qquad y^{2}-Vy+1=0.
\end{equation}
\end{definition}

\begin{proposition}[Elimination to a plane curve $G(U,V)=0$]\label{prop:G}
Let $(s,y)\in(\mathbb{Q}\setminus\{0\})^2$ satisfy $F(s,y)=0$.
Then $(U,V)\in\mathbb{Q}^2$ defined by Definition~\ref{def:UV} satisfies
\begin{equation}\label{eq:G}
G(U,V)=0,
\end{equation}
where $G\in\mathbb{Z}[U,V]$ is the degree-$5$ polynomial in $V$
\begin{equation}\label{eq:Gpoly}
\begin{aligned}
G(U,V)=\;&V^{5}+(4U-2)V^{4}+(-10U^{2}-8U+64)V^{3}\\
&+(4U^{3}-108U^{2}+384)V^{2}+(U^{4}-8U^{3}-192U^{2}+768)V\\
&+(-2U^{4}-128U^{2}+512).
\end{aligned}
\end{equation}
\end{proposition}

\begin{proof}
By \eqref{eq:quadUV}, the existence of $(s,y)$ with given $(U,V)$ is equivalent to the vanishing of the double resultant
\[
\mathrm{Res}_{y}\!\left(\mathrm{Res}_{s}\bigl(F(s,y),\,s^{2}-Us+1\bigr),\,y^{2}-Vy+1\right)\in\mathbb{Z}[U,V].
\]
A direct elimination yields that this resultant equals $G(U,V)^2$ (no additional factors occur), and hence $G(U,V)=0$ for any solution $(s,y)$.
\end{proof}

\begin{proposition}[Rationality and parametrization]\label{prop:param}
The affine curve $G(U,V)=0$ has a singular point at $(U,V)=(2,-2)$ of multiplicity $4$, hence it is a rational curve.
Moreover, the family of lines through $(2,-2)$
\[
V+2=t(U-2),\qquad t\in\mathbb{P}^{1},
\]
parametrizes $G=0$ as
\begin{equation}\label{eq:UVparam}
U=U(t),\qquad V=V(t),
\end{equation}
where
\begin{equation}\label{eq:Uparam}
U(t)=\frac{2\,(t^{5}+6t^{4}+30t^{3}+16t^{2}+9t+2)}{t\,(t-1)^{2}(t^{2}+6t+1)},
\end{equation}
\begin{equation}\label{eq:Vparam}
V(t)=\frac{2\,(t^{4}+36t^{3}+22t^{2}+4t+1)}{(t-1)^{2}(t^{2}+6t+1)}.
\end{equation}
\end{proposition}

\begin{proof}
Substitute $V=t(U-2)-2$ into $G(U,V)$.
A direct expansion and factorization gives
\[
G\bigl(U,\,t(U-2)-2\bigr)=(U-2)^{4}\cdot\Bigl(U\cdot\bigl(t^{5}+4t^{4}-10t^{3}+4t^{2}+t\bigr)-\bigl(2t^{5}+12t^{4}+60t^{3}+32t^{2}+18t+4\bigr)\Bigr).
\]
This shows that $(U,V)=(2,-2)$ is a multiplicity-$4$ singular point and that the remaining intersection is governed by a linear equation in $U$.
Solving for $U$ yields \eqref{eq:Uparam}, and then $V=t(U-2)-2$ gives \eqref{eq:Vparam}.
\end{proof}


\section{Lifting back to \texorpdfstring{$(s,y)$}{(s,y)} and a genus-\texorpdfstring{$2$}{2} obstruction curve}\label{sec:hyper}

The parametrization \eqref{eq:UVparam} describes all rational points on the quotient curve $G(U,V)=0$.
However, we must also enforce that $U$ and $V$ lift to rational $s$ and $y$ via \eqref{eq:quadUV}.

\begin{lemma}[Square conditions]\label{lem:squares}
Let $(s,y)\in(\mathbb{Q}\setminus\{0\})^2$ and define $(U,V)$ as in Definition~\ref{def:UV}.
Then
\[
U^{2}-4=\left(s-\frac1s\right)^{2}\in(\mathbb{Q}^{\times})^{2}\cup\{0\},\qquad
V^{2}-4=\left(y-\frac1y\right)^{2}\in(\mathbb{Q}^{\times})^{2}\cup\{0\}.
\]
Conversely, given $U,V\in\mathbb{Q}$, the quadratics \eqref{eq:quadUV} have solutions $s,y\in\mathbb{Q}$ if and only if $U^{2}-4$ and $V^{2}-4$ are squares in $\mathbb{Q}$.
\end{lemma}

\begin{proof}
This follows by completing the square in \eqref{eq:quadUV}.
\end{proof}

\begin{proposition}[Reduction to a genus-$2$ hyperelliptic curve]\label{prop:hyper}
Let $(s,y)\in(\mathbb{Q}\setminus\{0\})^{2}$ satisfy $F(s,y)=0$ and assume $s\neq 1$.
Then there exists $t\in\mathbb{Q}$ and $w\in\mathbb{Q}$ such that
\begin{equation}\label{eq:C}
w^{2}=t^{5}+21t^{4}+26t^{3}+10t^{2}+5t+1.
\end{equation}
Equivalently, there exists a rational point on the genus-$2$ hyperelliptic curve
\begin{equation}\label{eq:Ccurve}
C:\quad w^{2}=(t+1)(t^{4}+20t^{3}+6t^{2}+4t+1).
\end{equation}
\end{proposition}

\begin{proof}
Let $U=s+1/s$ and $V=y+1/y$.
By Proposition~\ref{prop:G} we have $G(U,V)=0$.
Since $s\neq 1$, we have $U\neq 2$ and therefore we can define
\[
t:=\frac{V+2}{U-2}\in\mathbb{Q}.
\]
By Proposition~\ref{prop:param}, this $t$ corresponds to the line through the singular point $(2,-2)$ and $(U,V)$, hence $(U,V)$ is represented by $U=U(t)$ and $V=V(t)$.

Now compute (using \eqref{eq:Uparam}--\eqref{eq:Vparam}) that
\begin{equation}\label{eq:Udisc}
U(t)^{2}-4=\frac{16\,(t+1)^{5}(t^{4}+20t^{3}+6t^{2}+4t+1)}{t^{2}(t-1)^{4}(t^{2}+6t+1)^{2}},
\end{equation}
\begin{equation}\label{eq:Vdisc}
V(t)^{2}-4=\frac{256\,t^{2}(t+1)(t^{4}+20t^{3}+6t^{2}+4t+1)}{(t-1)^{4}(t^{2}+6t+1)^{2}}.
\end{equation}
The denominators in \eqref{eq:Udisc} and \eqref{eq:Vdisc} are squares in $\mathbb{Q}$, and so are the factors $16$, $256$, $(t+1)^4$, and $t^2$.
Therefore $U^{2}-4$ and $V^{2}-4$ being squares (Lemma~\ref{lem:squares}) implies that
\[
(t+1)(t^{4}+20t^{3}+6t^{2}+4t+1)
\]
is a square in $\mathbb{Q}$, i.e.\ there exists $w\in\mathbb{Q}$ satisfying \eqref{eq:Ccurve}.
Expanding the right-hand side gives \eqref{eq:C}.
\end{proof}

\begin{remark}[Special values $t=0$ and $t=1$]
The parametrization \eqref{eq:Uparam}--\eqref{eq:Vparam} has poles at $t=0$ and $t=1$; geometrically these correspond to directions through the singular point $(2,-2)$ for which the line does not meet the curve $G=0$ at any other finite point.
In our application $s\neq 1$ implies $U\neq 2$, so $t=(V+2)/(U-2)$ is well-defined and does not correspond to the singular point.
\end{remark}


\section{Rational points on \texorpdfstring{$C$}{C} and closure of the \texorpdfstring{$1+4$}{1+4} case}\label{sec:points}

\begin{proposition}[Rational points on $C$]\label{prop:CQ}
Let $C$ be the hyperelliptic curve \eqref{eq:Ccurve}.
Then
\[
C(\mathbb{Q})=\bigl\{\infty,\ (-1,0),\ (0,\pm 1),\ (1,\pm 8)\bigr\}.
\]
\end{proposition}

\begin{proof}[Computational proof in \textsc{Magma}]
This is certified by Script~02 in Appendix~\ref{app:scripts}.
The script computes a rank bound $\mathrm{RankBound}(J)=1$ for the Jacobian $J=\mathrm{Jac}(C)$, constructs a Jacobian point of infinite order (verified by \texttt{Order(pt\_J) eq 0}), and then applies \texttt{Chabauty(pt\_J)} (based on the method of Chabauty~\cite{Chabauty} and Coleman~\cite{Coleman}) to obtain all rational points on $C$.
The complete output list agrees with the set above.
\end{proof}

\begin{lemma}[No admissible parameters from $C(\mathbb{Q})$]\label{lem:noadmissible}
Let $(t,w)\in C(\mathbb{Q})$ and let $U=U(t)$ be as in \eqref{eq:Uparam} whenever defined.
Then the only positive rational value of $s$ satisfying $U=s+1/s$ is $s=1$.
\end{lemma}

\begin{proof}
By Proposition~\ref{prop:CQ}, the rational points on $C$ correspond to $t \in \{-1, 0, 1, \infty\}$. We determine the admissible values of $s$ by solving $s+1/s = U(t)$ using the expression \eqref{eq:Uparam}.

\smallskip
\noindent\emph{Case $t=\infty$.}
Comparing the degrees of the numerator and denominator in \eqref{eq:Uparam}, we find $\lim_{t\to\infty} U(t) = 2$. The equation $s+1/s=2$ is equivalent to $(s-1)^2=0$, yielding the unique solution $s=1$. This corresponds to the excluded case $p=q$.

\smallskip
\noindent\emph{Case $t=-1$.}
Substituting $t=-1$ into \eqref{eq:Uparam} yields
\[
U(-1) = \frac{2\,(-1+6-30+16-9+2)}{(-1)\,(-2)^{2}\,(1-6+1)} = \frac{-32}{16} = -2.
\]
The equation $s+1/s=-2$ implies $(s+1)^2=0$, so $s=-1$. However, the cuboid parameter $s=(p/q)^2$ must be positive, so this solution is inadmissible.

\smallskip
\noindent\emph{Cases $t=0$ and $t=1$.}
The denominator of $U(t)$ in \eqref{eq:Uparam} contains the factors $t$ and $(t-1)^2$, so it vanishes at $t=0$ and $t=1$. Evaluating the numerator at these points yields $4$ (for $t=0$) and $128$ (for $t=1$), which are non-zero.
Consequently, $t=0$ and $t=1$ are poles of the rational function $U(t)$, implying $|U(t)| \to \infty$.
The relation $s+1/s=U$ implies that $s \to 0$ or $s \to \infty$. Since we require $s \in \mathbb{Q}_{>0}$ (a finite non-zero rational), these points yield no admissible parameters.

\smallskip
\noindent\emph{Conclusion.}
The only positive rational $s$ arising from $C(\mathbb{Q})$ is $s=1$.
\end{proof}

\begin{theorem}[No $1+4$ factorization for $P_s$; no rational roots of $Q_{p,q}(t)$]\label{thm:main}
Let $s\in\mathbb{Q}_{>0}$ with $s\neq 1$.
Then the quintic $P_s(x)$ has no rational root (equivalently, it admits no $1+4$ factorization over $\mathbb{Q}$).

In particular, for coprime $p,q\in\mathbb{Z}_{>0}$ with $p\neq q$, setting $s=(p/q)^2$, the second cuboid polynomial $Q_{p,q}(t)$ has no rational roots (hence no linear factor over $\mathbb{Q}$, and in particular no linear factor over $\mathbb{Z}$).
\end{theorem}

\begin{proof}
Assume, for contradiction, that $P_s(x)$ has a rational root for some $s\in\mathbb{Q}_{>0}$ with $s\neq 1$.
Let $x=sy$ as in Lemma~\ref{lem:F}.
Then $F(s,y)=0$ for some $y\in\mathbb{Q}$.
By Proposition~\ref{prop:hyper} (using $s\neq 1$) there exists a rational point $(t,w)\in C(\mathbb{Q})$ on the hyperelliptic curve \eqref{eq:Ccurve}.
By Proposition~\ref{prop:CQ} this point is among $\{\infty,(-1,0),(0,\pm1),(1,\pm8)\}$.
Lemma~\ref{lem:noadmissible} shows that the only positive rational $s$ arising from these points is $s=1$, contradicting $s\neq 1$.

Therefore $P_s$ has no rational root for $s>0$ with $s\neq 1$.

For the stated consequence for $Q_{p,q}(t)$, suppose that $Q_{p,q}(t)$ had a rational root $t_0\in\mathbb{Q}$.
Then by Lemma~\ref{lem:14_28} (final sentence) the corresponding normalized parameter $s=(p/q)^2$ would yield a rational root of $P_s$, contradicting what we have just proved.
Hence $Q_{p,q}(t)$ has no rational roots.
\end{proof}


\section*{Acknowledgments}
The author would like to thank Randall L. Rathbun for helpful discussions and correspondence regarding the problem.


\appendix
\section{\textsc{Magma} scripts and transcripts}\label{app:scripts}

\noindent
All computer-assisted steps in this note are executed in \textsc{Magma}.
Script~01 is a diagnostic computation showing that the naive plane curve $S(x,r)=0$ (encoding rational roots directly) has genus $6$.
Script~02 is the certificate computation used in the main proof: it computes the rational points of the genus-$2$ curve $C$ via a rank bound and Chabauty on the Jacobian.

\subsection*{Script 01: Genus of the naive rational-root curve $S(x,r)=0$}
\noindent\textbf{Code.}
\begin{code}
// Setup Ring and Polynomial
Q := Rationals();
R<x, r> := PolynomialRing(Q, 2);

// Coefficients
A := (2 + r^2)*(3 - 2*r^2);
B := 1 + 10*r^2 + 4*r^4 - 14*r^6 + r^8;
C := -r^2*(1 - 14*r^2 + 4*r^4 + 10*r^6 + r^8);
D := -r^6*(1 + 2*r^2)*(-2 + 3*r^2);
Const := -r^10;

// The polynomial S(x, r) itself
S := x^5 + A*x^4 + B*x^3 + C*x^2 + D*x + Const;
print "Polynomial S(x, r) constructed.";

// If S(x) has a root, then the point (x,r) lies on the curve S(x,r) = 0.
// Define Affine Space and Curve
A2 := AffineSpace(R);
C_aff := Curve(A2, S);

// Take Projective Closure to compute the genus correctly
C_proj := ProjectiveClosure(C_aff);

// Compute the geometric genus, resolving singularities internally if the curve is singular.
g_root := Genus(C_proj);

printf "Curve equation defined by S(x, r) = 0\n";
printf "Geometric Genus (1+4 case): %o\n", g_root;
\end{code}

\noindent\textbf{Transcript.}
\begin{term}
Polynomial S(x, r) constructed.
Curve equation defined by S(x, r) = 0
Geometric Genus (1+4 case): 6
\end{term}

\subsection*{Script 02: Rational points on $C$ via Jacobian rank bound and Chabauty}
\noindent\textbf{Code.}
\begin{code}
// Setup Polynomial ring and Hyperelliptic curve
Q<t> := PolynomialRing(Rationals());
f := t^5 + 21*t^4 + 26*t^3 + 10*t^2 + 5*t + 1;
C := HyperellipticCurve(f);
J := Jacobian(C);

// 1. Find "naive" rational points (heuristic search)
pts := RationalPoints(C : Bound := 1000);
print "Found points:", pts;

// 2. Compute the Rank Bound of the Jacobian
r := RankBound(J);
print "Rank Bound:", r; 

if r eq 1 then
    print "Rank is 1. Using Chabauty (requires a generator)...";
    
    // We need a point of infinite order on the Jacobian.
    // We take the difference of two points on the curve: P_rat - P_inf
    
    // Point (1:0:0) represents Infinity on this model.
    P_inf := C![1, 0, 0]; 
    
    // Select a rational point (0 : 1 : 1) from the found list 'pts'
    // Note: Avoid points with y=0 (Weierstrass points) as they are torsion.
    P_rat := C![0, 1, 1]; 
    
    // Construct a point on the Jacobian: D = [P_rat - P_inf]
    // Note: We use a sequence [P1, P2] to define the divisor class P1-P2
    pt_J := J ! [P_rat, P_inf];
    
    // Verify the point is not torsion (Order must be 0 for infinite order)
    if Order(pt_J) eq 0 then
        // Run classical Chabauty
        final_points := Chabauty(pt_J);
        print "All proven rational points:", final_points;
    else
        print "Error: Selected point has finite order (torsion). Try a different P_rat.";
    end if;

elif r eq 0 then
    print "Rank is 0. Using Chabauty0...";
    // Chabauty0 is specifically for Rank 0 cases
    print Chabauty0(J);
else
    print "Rank >= Genus. Chabauty method is not applicable.";
end if;
\end{code}

\noindent\textbf{Transcript.}
\begin{term}
Found points: {@ (1 : 0 : 0), (-1 : 0 : 1), (0 : -1 : 1), (0 : 1 : 1), (1 : -8 :
1), (1 : 8 : 1) @}
Rank Bound: 1
Rank is 1. Using Chabauty (requires a generator)...
All proven rational points: { (1 : -8 : 1), (0 : -1 : 1), (1 : 8 : 1), (-1 : 0 :
1), (0 : 1 : 1), (1 : 0 : 0) }
\end{term}



\begingroup
\footnotesize
\begin{thebibliography}{99}

\bibitem{Sharipov2011Cuboids}
R.~A.~Sharipov,
\newblock \emph{Perfect cuboids and irreducible polynomials},
\newblock Ufa Math.\ J.\ \textbf{4} (2012), no.~1, 153--160.

\bibitem{Sharipov2011Note}
R.~A.~Sharipov,
\newblock \emph{Asymptotic approach to the perfect cuboid problem},
\newblock Ufa Math.\ J.\ \textbf{7} (2015), no.~3, 95--107.

\bibitem{Magma}
W.~Bosma, J.~Cannon, and C.~Playoust,
\newblock The Magma algebra system. {I}. The user language,
\newblock J.\ Symbolic Comput.\ \textbf{24} (1997), 235--265.

\bibitem{Chabauty}
C.~Chabauty,
\newblock Sur les points rationnels des courbes alg{\'e}briques de genre sup{\'e}rieur \`{a} l'unit{\'e},
\newblock C.\ R.\ Acad.\ Sci.\ Paris \textbf{212} (1941), 882--885.

\bibitem{Coleman}
R.~F.~Coleman,
\newblock Effective Chabauty,
\newblock Duke Math.\ J.\ \textbf{52} (1985), no.~3, 765--770.

\end{thebibliography}
\endgroup
\end{document}